\documentclass[11pt]{article}
\usepackage{amsmath}
\usepackage{amsthm}
\usepackage{hyperref}
\usepackage{multirow}

\newtheorem{theorem}{Theorem}[section]
\newtheorem{obs}[theorem]{Observation}
\newtheorem{rem}[theorem]{Remark}
\newtheorem{lem}[theorem]{Lemma}
\newtheorem*{urem}{Remark}
\newtheorem*{urems}{Remarks}
\newtheorem{propos}[theorem]{Proposition}
\newtheorem*{udefin}{Defininiton}
\newtheorem{cor}[theorem]{Corollary}

\begin{document}

\title{Anti-Ramsey numbers of graphs with small connected components}

\author{Shoni Gilboa\thanks{Mathematics Dept., The Open University of Israel, Raanana 43107, Israel. \texttt{tipshoni@gmail.com} Tel: 972-77-7081316}
\and
Yehuda Roditty\thanks{Schools of Computer Sciences, The Academic College of Tel-Aviv-Yaffo, and Tel-Aviv University, Tel-Aviv 69978, Israel. \texttt{jr@mta.ac.il} 
      } 
}

\maketitle

\begin{abstract} The anti-Ramsey number, $AR(n,G)$, for a graph $G$ and an integer $n\geq|V(G)|$, is defined to be the minimal integer $r$ such that in any edge-colouring of $K_n$ by at least $r$ colours there is a multicoloured copy of $G$, namely, a copy of $G$ that each of its edges has a distinct colour. In this paper we determine, for large enough $n$, $AR(n,L\cup tP_2)$ and $AR(n,L\cup kP_3)$ for any large enough $t$ and $k$, and a graph $L$ satisfying some conditions.  
Consequently, we determine $AR(n,G)$, for large enough $n$, where $G$ is $P_3\cup tP_2$ for any $t\geq 3$, $P_4\cup tP_2$ and $C_3\cup tP_2$ for any $t\geq 2$, $kP_3$ for any $k\geq 3$, $tP_2\cup kP_3$ for any $t\geq 1$, $k\geq 2$, and $P_{t+1}\cup kP_3$ for any $t\geq 3$, $k\geq 1$. 
Furthermore, we obtain upper and lower bounds for $AR(n,G)$,  for large enough $n$, where $G$ is $P_{k+1}\cup tP_2$ and $C_k\cup tP_2$ for any $k\geq 4$, $t\geq 1$. 
\end{abstract}

\begin{quote}
\end{quote}

\section{Introduction}\label{intro}

\begin{udefin} A subgraph of an edge-coloured graph is called {\bf multicoloured} if each of its edges has a distinct colour.

Let $G$ be a (simple) graph. For any integer $n\geq|V(G)|$, let $AR(n,G)$ be the minimal integer $r$ such that in any edge-colouring of $K_n$ by at least $r$ colours there is a multicoloured copy of $G$.
\end{udefin}
$AR(n,G)$ was determined for various graphs $G$. We mention some of the results, which are relevant to our work.

For $P_{k+1}$, a path of length $k\geq 2$, Simonovits and S\'os showed (\cite{SS}) that for large enough $n$ ($n\geq\frac{5}{4}k+c$ for some universal constant $c$),
\begin{equation}\label{eq:P_k}AR(n,P_{k+1})=(\lfloor k/2\rfloor-1)\left(n-\frac{\lfloor k/2\rfloor}{2}\right)+2+k\bmod 2.\end{equation}

For $C_k$, a cycle of length $k$, Montellano-Ballesteros and Neumann-Lara (\cite{MbNl}) proved that for any $n\geq k\geq 3$,
\begin{equation}\label{eq:C_k}AR(n,C_k)=\binom{k-1}{2}\left\lfloor\frac{n}{k-1}\right\rfloor+\left\lceil\frac{n}{k-1}\right\rceil+\binom{n\bmod(k-1)}{2},\end{equation}
after Erd\H os, Simonovits and S\'os noted in \cite{ESS}, where anti-Ramsey numbers were first introduced, that \eqref{eq:C_k} holds for $n\geq k=3$, showed the lower bound in \eqref{eq:C_k} for any $n\geq k\geq 3$ and conjectured this lower bound to be always tight, and Alon proved (\cite{A}) that \eqref{eq:C_k} holds for $n\geq k=4$.

For $tP_2$, the disjoint union of $t$ paths of length $1$, i.e., a matching of size $t$, Schiermeyer first showed (\cite{S}) that $AR(n,tP_2)=(t-2)\left(n-\frac{t-1}{2}\right)+2$
for $t\geq 2$, $n\geq 3t+3$.
Then Fujita, Kaneko, Schiermeyer and Suzuki proved (\cite{FKSS}) that for any $t\geq 2$, $n\geq 2t+1$,
\begin{equation}\label{eq:tP_2}AR(n,tP_2)=\begin{cases}(t-2)(2t-3)+2&n\leq\frac{5t-7}{2}\\(t-2)\left(n-\frac{t-1}{2}\right)+2&n\geq\frac{5t-7}{2}\,.\end{cases}\end{equation}
Finally, the remaining case $n=2t$ was settled by Haas and Young (\cite{HY}) who confirmed the conjecture made in \cite{FKSS}, that 
\begin{equation}\label{eq:perfect}AR(2t,tP_2)=\begin{cases}(t-2)\frac{3t+1}{2}+2&3\leq t\leq 6\\(t-2)(2t-3)+3&t\geq 7\,.\end{cases}\end{equation}
\smallskip

In Section \ref{sec:tP2} we prove the following theorem which enables to transfer any linear upper bound on $AR(n,L\cup t_1P_2)$ (for large enough $n$) to a linear upper bound on $AR(n, L\cup tP_2)$ (for large enough $n$) for any $t>t_1$.
\medskip

\noindent{\bf Theorem \ref{thm:tP2}.} {\it  Let $L$ be a graph, let $t_1\geq 0$ and $n_0\geq |V(L)|+2t_1$ be integers, and let $r$ and $s$ be real numbers. Suppose that $AR(n,L\cup t_1P_2)\leq (t_1+r)\left(n-\frac{t_1+r+1}{2}\right)+s+1$ for any integer $n\geq n_0$.

Then, there is a constant $\gamma_2$, depending only on $L$, $t_1$, $r$, $s$ and $n_0$, such that for any integers $t\geq t_1$ and $n>\frac{5}{2}t+\gamma_2$,} 
\begin{equation*}\label{eq:LtP2}AR(n,L\cup tP_2)\leq (t+r)\left(n-\frac{t+r+1}{2}\right)+s+1.\end{equation*} 
For $L$ satisfying some additional restrictions we show, in Proposition \ref{prop:tP2tight}, that the upper bound of Theorem \ref{thm:tP2} is actually tight.
Using Proposition \ref{prop:tP2tight} we then easily get that for large enough $n$,
\begin{align*}
AR(n,P_3\cup tP_2)=(t-1)\left(n-\frac{t}{2}\right)&+2\quad(t\geq 2,\text{Corollary }\ref{cor:PktP2}),\\
AR(n,P_4\cup tP_2)=t\left(n-\frac{t+1}{2}\right)&+2\quad (t\geq 1,\text{Corollary }\ref{cor:PktP2}),\\
AR(n,C_3\cup tP_2)=t\left(n-\frac{t+1}{2}\right)&+2\quad (t\geq 1,\text{Corollary }\ref{cor:C3tP2}).
\end{align*}
We also provide upper and lower bounds for $AR(n,P_{k+1}\cup tP_2)$ and $AR(n, C_k\cup tP_2)$ for any $k\geq 4$ and $t\geq 1$.

In Section \ref{sec:tP3}, we prove the following Theorem, analogous to Theorem \ref{thm:tP2},
which enables to transfer any linear upper bound on $AR(n,L\cup k_1P_3)$ (for large enough $n$) to a linear upper bound on $AR(n, L\cup kP_3)$ (for large enough $n$) for any $k>k_1$.
\medskip

\noindent{\bf Theorem \ref{thm:LtP3}.} {\it Let $L$ be a graph, let $k_1\geq 0$ and $n_0\geq |V(L)|+3k_1$ be integers, and let $r$ and $s$ be real numbers. Suppose that $AR(n,L\cup k_1P_2)\leq (k_1+r)\left(n-\frac{k_1+r+1}{2}\right)+s+1$ for any integer $n\geq n_0$.

Then, there is a constant $\gamma_3$, depending only on $L$, $k_1$, $r$, $s$ and $n_0$, such that for any integers $k\geq k_1$ and $n>5k+\gamma_3$,}
\begin{equation*}AR(n,L\cup kP_3)\leq (k+r)\left(n-\frac{k+r+1}{2}\right)+s+1.\end{equation*} 
\newline 
This theorem enables us to show that for large enough $n$,
\begin{align*}
AR(n,kP_3)&=(k-1)\left(n-\frac{k}{2}\right)+2&(k\geq 1,\text{Corollary }\ref{cor:tP3}),\\
AR(n,P_{t+1}\cup kP_3)&=(k+\lfloor t/2\rfloor-1)\left(n-\frac{k+\lfloor t/2\rfloor}{2}\right)+2+t\bmod 2&(t\geq 3,k\geq 0,\text{Corollary }\ref{cor:PtkP3}),\\
AR(n,P_2\cup kP_3)&=(k-1)\left(n-\frac{k}{2}\right)+3&(k\geq 1,\text{Corollary }\ref{cor:P2kP3}),\\
AR(n,tP_2\cup kP_3)&=(k+t-2)\left(n-\frac{k+t-1}{2}\right)+2&(t\geq 2,k\geq 2,\text{Corollary }\ref{cor:tP2kP3}).\\
\end{align*}

To get some of the consequences, mentioned above,  of Theorems \ref{thm:tP2} and \ref{thm:LtP3}, we use upper bounds on the anti-Ramsey numbers of some small graphs. Those upper bounds are taken from  \cite{BGR}, where a complete account of the anti-Ramsey numbers of graphs with no more than four edges is given.

\section{Notation}

\begin{itemize}
\item Let $G=(V,E)$ be a (simple) graph.
\begin{enumerate}
\item For any (not necessarily disjoint) sets $A,B\subseteq V$ let $E_G(A,B):=\{uv\in E \mid u\in A,v\in B\}$. 
\item For each $v \in V$ let $N_G(v):= \{ w \in V \mid vw \in E \}$, and $d_G(v):=|N_G(v)|$.
\end{enumerate} 
\item The complete graph on a vertex set $V$ will be denoted $K^V$.
\item Let $c$ be an edge-colouring of a graph $G=(V,E)$.
\begin{enumerate}
\item We denote by $c(uv)$ the colour an edge $uv$ has. 
\item For any $v\in V$ let $C(v):=\{c(vw) \mid w\in N_G(v)\}$ and $d_c(v):=|C(v)|$.
\end{enumerate}
\end{itemize}

\section{The anti-Ramsey numbers of $L\cup tP_2$}\label{sec:tP2}

\begin{theorem}\label{thm:tP2} Let $L$ be a graph, let $t_1\geq 0$ and $n_0\geq |V(L)|+2t_1$ be integers, and let $r$ and $s$ be real numbers. Suppose that $AR(n,L\cup t_1P_2)\leq (t_1+r)\left(n-\frac{t_1+r+1}{2}\right)+s+1$ for any integer $n\geq n_0$.

Then, there is a constant $\gamma_2$, depending only on $L$, $t_1$, $r$, $s$ and $n_0$, such that for any integers $t\geq t_1$ and $n>\frac{5}{2}t+\gamma_2$, 
\begin{equation*}\label{eq:LtP2}AR(n,L\cup tP_2)\leq (t+r)\left(n-\frac{t+r+1}{2}\right)+s+1.\end{equation*} 
\end{theorem}

\begin{proof} 
The proof is by induction on $t$. The base case, $t=t_1$, is obvious (provided $\frac{5}{2}t_1+\gamma_2\geq n_0-1$). 

Now let $t>t_1$, and assume that $AR(n,L\cup (t-1)P_2)\leq (t-1+r)\left(n-\frac{t-1+r+1}{2}\right)+s+1$ for any $n>\frac{5}{2}(t-1)+\gamma_2$. Let $c$ be any edge-coloring of $K^V$, where $|V|=n>\frac{5}{2}t+\gamma_2$, by at least $(t+r)\left(n-\frac{t+r+1}{2}\right)+s+1$ colours. We will find a multicoloured copy of $L\cup tP_2$ in $K^V$. Let $\ell:=|V(L)|$. The proof is divided into two cases.
\medskip\newline {\bf Case 1.} $d_c(v_0)\geq 2t+\ell$ for some vertex $v_0$.
\medskip\newline 
Changing the colour of every edge $e$ of $K^{V-\{v_0\}}$ for which $c(e)\in C(v_0)$ (if there are any such edges) to some common colour $c_0$, we get an edge-colouring $c^*$ of $K^{V-\{v_0\}}$ by at least 
$$(t+r)\left(n-\frac{t+r+1}{2}\right)+s+1-(n-1)=(t-1+r)\left(n-1-\frac{t-1+r+1}{2}\right)+s+1$$
colours. Since clearly $n-1>\frac{5}{2}t+\gamma_2-1>\frac{5}{2}(t-1)+\gamma_2$, we get by the induction hypothesis that $K^{V-\{v_0\}}$ contains a copy $G$ of $L\cup (t-1)P_2$ which is multicoloured with respect to $c^*$, and therefore also according to the original colouring $c$.

The vertex $v_0$ is the endpoint of at least $2t+\ell$ edges with distinct colours (with respect to $c$). The other endpoint of at most $2(t-1)+\ell$ of those edges is a vertex of $G$. Also, at most one of those edges have the same colour, according to $c$, as an edge of $G$ (since at most one of the edges of $G$ is coloured by the colour $c_0$ according to $c^*$, i.e., by a colour in $C(v_0)$ with respect to $c$). Therefore we are surely left with at least one edge $v_0w$ such that $w\notin V(G)$ and $c(v_0w)\notin\{c(e)\mid e\in E(G)\}$. By adding such an edge to $G$ we get the desired muticoloured copy of $L\cup tP_2$.  
\medskip\newline {\bf Case 2.} $d_c(v)\leq 2t+\ell-1$ for all $v\in V$.
\medskip\newline
By the induction hypothesis $K^V$ clearly contains a multicoloured copy, $G$, of $L\cup (t-1)P_2$. Assume, by contradiction, that $K^V$ does not contain a multicoloured copy of $L\cup tP_2$.

Form a graph $H$ on the vertex set $V$ by adding to the edges of $G$ a single edge of each colour of $c$ not represented in $G$. By our assumptions, $d_H(v)\leq 2t+\ell-1$ for any $v\in V$, and $H$ does not contain a copy of $L\cup tP_2$.  

Let $U_L$ be the vertex set of the $L$ part of $G$, $U$ the vertex set of the $(t-1)P_2$ part of $G$, and let $W:=V- (U_L\cup U)$. 
Call a vertex $u\in U$ {\bf fat} if $\lvert E_H(\{u\},W)\rvert\geq 2$, and {\bf thin} otherwise.

If $uv$ is an edge of $G$, and $u\in U$ is fat, then $E_H(\{v\},W)=\emptyset$. (Otherwise, we could get from $G$ a copy of $L\cup tP_2$ in $H$ by replacing the edge $uv$ by two edges, one connecting $v$ to some $w\in N_H(v)\cap W$, and the other connecting $u$ to some vertex, different than $w$, in $N_H(u)\cap W$). In particular, any edge of $G$ has at most one fat endpoint.

Let $F\subseteq U$ be the set of fat vertices, $N\subseteq U$ the set of thin vertices such that their (only) neighbour in $G$ is fat, and $T\subseteq U$ the set of all other thin vertices in $U$. Notice that 
\begin{equation}\label{FTN}|F|+\frac{|T|}{2}=|N|+\frac{|T|}{2}=t-1.\end{equation}

The set $N$ is an independent set in $H$. (Otherwise, if there were vertices $u_1, u_2$ in $N$ adjacant in $H$, we could get from $G$ a copy of $L\cup tP_2$ in $H$ by replacing the two edges of $G$ containing $u_1, u_2$ by the edge $u_1u_2$ and for $i=1,2$, an edge between the fat neighbour of $u_i$ in $G$ and one of its neighbours in $W$). 
The set $W$ is also an independent set in $H$, otherwise we could get a copy of $L\cup tP_2$ in $H$ by adding to $G$ an edge from $E_H(W,W)$.
Since $E_H(\{v\},W)=\emptyset$ for any $v\in N$, it follows that $E_H(N,W)=\emptyset$. 
Therefore
\begin{equation}\label{EVV}\lvert E_H(V,V)\rvert=\lvert E_H(N,T)\rvert+\lvert E_H(T,T)\rvert+\lvert E_H(T,W)\rvert+\lvert E_H(F\cup U_L,V)\rvert.\end{equation}
By the definition of thin vertices,
\begin{equation}\label{ETW}\lvert E_H(T,W)\rvert\leq |T|=2(t-1-|F|),\end{equation}
and by the assumption that $\Delta(H)\leq 2t+\ell-1$, we have
\begin{equation}\label{EFV}\lvert E_H(F\cup U_L,V)\rvert\leq |F\cup U_L|(2t+\ell-1)-|E(L)|=(|F|+\ell)(2t+\ell-1)-|E(L)|.\end{equation}  
Substituting \eqref{ETW} and \eqref{EFV} in \eqref{EVV} and using \eqref{FTN} we get
\begin{align*}\lvert E_H(V,V)\rvert&=\lvert E_H(N,T)\rvert+\lvert E_H(T,T)\rvert+\lvert E_H(T,W)\rvert+\lvert E_H(F\cup U_L,V)\rvert\leq\\
&\leq |N|\cdot |T|+\binom{|T|}{2}+|T|+(|F|+\ell)(2t+\ell-1)-|E(L)|=\\
&=\left(|N|+\frac{|T|-1}{2}+1\right)|T|+(2t-1)(|F|+\ell)+\ell(|F|+\ell)-|E(L)|=\\
&=(2t-1)\left(\frac{|T|}{2}+|F|+\ell\right)+\ell(|F|+\ell)-|E(L)|=\\
&=(2t-1)(t-1+\ell)+\ell(|F|+\ell)-|E(L)|\leq(2t-1+\ell)(t-1+\ell)-|E(L)|.
\end{align*}
Therefore
$$(t+r)\left(n-\frac{t+r+1}{2}\right)+s+1\leq (2t-1+\ell)(t-1+\ell)-|E(L)|.$$
After some rearranging we get
$$(t+r)\left(n-\frac{5}{2}t+\frac{3}{2}r-3\ell+\frac{5}{2}\right)\leq (\ell-1)^2-3(\ell-1)r+2r^2-|E(L)|-s-1,$$
yielding a contradiction for $n>\frac{5}{2}t+\gamma_2$, if we take $\gamma_2$ such that
\begin{equation*}(t_1+1+r)\left(\gamma_2+\frac{3}{2}r-3\ell+\frac{5}{2}\right)\geq\max\left\{0\,,\, (\ell-1)^2-3(\ell-1)r+2r^2-|E(L)|-s-1\right\}.\qedhere\end{equation*}
\end{proof}

\begin{rem}\label{rem:gamma} As the proof above shows, $\gamma_2$ may be taken to be 
\begin{equation*}\max\left\{n_0-1-\frac{5}{2}t_1\,,\,3\ell-\frac{3}{2}r-\frac{5}{2}+\frac{(\ell-1)^2-3(\ell-1)r+2r^2-|E(L)|-s-1}{t_1+1+r}\right\}\end{equation*}
if $(\ell-1)^2-3(\ell-1)r+2r^2-|E(L)|-s-1\geq 0$, and
$\max\left\{n_0-1-\frac{5}{2}t_1\,,\,\frac{1}{2}\left\lceil 6\ell-3r-6\right\rceil\right\}$ otherwise.
\end{rem}
\medskip
When $L$ satisfies some additional restrictions, which will be described using the following definition, we can show, in Proposition \ref{prop:tP2tight} below, that the upper bound of Theorem \ref{thm:tP2} is actually tight.

\begin{udefin} For a graph $G=(V,E)$ and a non-negative integer $j$, let
$$q_j(G):=\min\{|R|\mid R\subseteq V\,,~|E_G(V- R,V-R)|\leq j\}.$$
Namely, $q_j(G)$ is the minimal size of a set of vertices incident with all but at most $j$ edges of $G$.
\end{udefin}
\begin{obs}\label{obs:q} Let $G=(V,E)$ be a graph, and let $r$ and $s$ be non-negative integers. If $\binom{|V|-r}{2}\leq s$ then $q_s(G)\leq r$.
( Since then $|E_G(V- R,V-R)|\leq\binom{|V|-r}{2}\leq s$ for {\bf any} set $R\subseteq V$ of cardinality $r$).
\end{obs}

\begin{lem}\label{lem:q} Let $s$ be a positive integer.
\begin{enumerate}
\item Let $G$ be a graph and let $r_1$ be a non-negative integer. If $q_s(G)>r_1$, then for any integer $n\geq |V(G)|$,
\begin{equation*}\label{eq:lem1}AR(n,G)>r_1\left(n-\frac{r_1+1}{2}\right)+s.\end{equation*}
\item Let $L$ be a graph, and let $t_2\geq 0$ and $r_2\geq - t_2$ be integers. If $q_{s-i}(L\cup t_2P_2)>t_2+r_2+i$ for any $0\leq i\leq s$, then for any integers $t\geq t_2$ and $n\geq 2t+|V(L)|$,
\begin{equation*}\label{eq:lemt}AR(n,L\cup tP_2)>(t+r_2)\left(n-\frac{t+r_2+1}{2}\right)+s.\end{equation*}
\end{enumerate}
\end{lem}
\begin{proof} 
To prove the first claim, let $V$ be the vertex set of $K_n$. Choose a set $R\subseteq V$ of cardinality $r_1$. By Observation \ref{obs:q}, 
$$\binom{|V-R|}{2}=\binom{n-r_1}{2}\geq\binom{|V(G)|-r_1}{2}>s.$$
Colour arbitrarily the edges of $K^{V- R}$ by exactly $s$ colours, and all other edges of $K_n$ by $r_1\left(n-\frac{r_1+1}{2}\right)$ distinct colours. Assume, by contradiction, that there is a multicoloured copy, $\tilde{G}$, of $G$. Then, 
$$|E_{\tilde{G}}(V(\tilde{G})- R,V(\tilde{G})- R)|\leq|E_{\tilde{G}}(V- R,V- R)|\leq s,$$
so $q_s(G)\leq|R\cap V(\tilde{G}|\leq|R|=r_1$ and we get a contradiction.

The second claim follows by applying the first claim to $G=L\cup t_2P_2$ and $r_1=t_2+r_2$, upon observing that for any $t\geq t_2$,
\begin{equation*}q_s(L\cup tP_2)=q_s\left((t-t_2)P_2\cup(L\cup t_2P_2)\right)=\min_{0\leq i\leq\min\{s,t-t_2\}}(t-t_2)-i+q_{s-i}(L\cup t_2P_2).
\qedhere
\end{equation*}
\end{proof}

Combining the upper bound of Theorem \ref{thm:tP2} and the lower bound of Lemma \ref{lem:q} we get:

\begin{propos}\label{prop:tP2tight}
Let $L$ be a graph and let $t_1,t_2\geq 0$, $r\geq -\min\{t_1,t_2\}$ and $s\geq 1$ be integers. Suppose that
\begin{itemize}
\item There is an integer $n_0\geq |V(L)|+2t_1$ such that $AR(n,L\cup t_1P_2)\leq (t_1+r)\left(n-\frac{t_1+r+1}{2}\right)+s+1$ for any integer $n\geq n_0$.
\item $q_{s-i}(L\cup t_2P_2)>t_2+r+i$ for any $0\leq i\leq s$.
\end{itemize}

Then, there is a constant $\gamma_2$, depending only on $L$, $t_1$, $r$, $s$ and $n_0$, such that for any integers $t\geq\max\{t_1,t_2\}$ and $n>\frac{5}{2}t+\gamma_2$, \begin{equation*}\label{eq:LtP2}AR(n,L\cup tP_2)=(t+r)\left(n-\frac{t+r+1}{2}\right)+s+1.\end{equation*} 
\end{propos}
\medskip
We now show several consequences of Proposition \ref{prop:tP2tight} and Theorem \ref{thm:tP2}.

\begin{cor}\label{cor:tP2}  
For any integers $t\geq 2$ and $n>\frac{5t+3}{2}$,
\begin{equation}\label{eq:tP2again}AR(n,tP_2)=(t-2)\left(n-\frac{t-1}{2}\right)+2.\end{equation}
\end{cor}
\begin{proof}
It is easy to see that $AR(n,2P_2)=2$ for any $n\geq 5$ (see \cite[Lemma 3.1]{BGR}), and clearly $q_1(2P_2)=1$ and $q_0(2P_2)=2$. The claim follows by taking $L$ to be the empty graph, $t_1=t_2=2$, $r=-2$, $s=1$ and $n_0=5$ in Proposition \ref{prop:tP2tight} and Remark \ref{rem:gamma}.
\end{proof}

\begin{urem}
As mentioned in the introduction, Fujita, Kaneko, Schiermeyer and Suzuki proved (\cite{FKSS}) that \eqref{eq:tP2again} holds for any $t\geq 2$, $n\geq\max\{2t+1,\frac{5t-7}{2}\}$, after Schiermeyer first showed (\cite{S}) it holds for $t\geq 2$, $n\geq 3t+3$.
\end{urem}

\begin{cor}\label{cor:PktP2} \begin{enumerate}
\item For any integers $t\geq 2$ and $n>\frac{5}{2}t+12$,
$$AR(n,P_3\cup tP_2) =(t-1)\left(n-\frac{t}{2}\right)+2.$$
\item For any integers $t\geq 1$ and $n\geq\frac{5}{2}t+12$,
$$AR(n,P_4\cup tP_2) =t\left(n-\frac{t+1}{2}\right)+2.$$
\item For any integers $k\geq 4$, $t\geq 0$ and $n\geq 2t+k+1$,
$$ AR(n,P_{k+1}\cup tP_2)\geq(t+\lceil k/2\rceil-2)\left(n-\frac{t+\lceil k/2\rceil-1}{2}\right)+2,$$
and for any integer $k\geq 4$ there is a constant $\gamma_2(P_{k+1})$ such that for any integers $t\geq 0$ and $n>\frac{5}{2}t+\gamma_2(P_{k+1})$,
$$AR(n,P_{k+1}\cup tP_2)\leq (t+\lfloor k/2\rfloor-1)\left(n-\frac{t+\lfloor k/2\rfloor}{2}\right)+2+k\bmod 2.$$
\end{enumerate}
\end{cor}
\begin{proof}
By \cite[Proposition 6.1]{BGR}, $AR(n,P_3\cup 2P_2)=n+1$ for any $n\geq 7$, and clearly $q_1(P_3\cup P_2)=1$ and $q_0(P_3\cup P_2)=2$. Taking $L=P_3$, $t_1=2$, $t_2=1$, $r=-1$, $s=1$ and $n_0=7$ in Proposition \ref{prop:tP2tight} and Remark \ref{rem:gamma}, we obtain the first part of the corollary. 

By \cite[Proposition 6.3]{BGR}, $AR(n,P_4\cup P_2)=n+1$ for any $n\geq 6$, and clearly $q_1(P_4)=1$ and $q_0(P_4)=2$. Taking $L=P_4$, $t_1=1$, $t_2=0$, $r=0$, $s=1$ and $n_0=6$ in Proposition \ref{prop:tP2tight} and Remark \ref{rem:gamma}, we obtain the second part of the corollary. 

For $k\geq 4$, the lower bound for $AR(n,P_{k+1}\cup tP_2)$  follows, since $q_1(P_{k+1})=\lfloor k/2\rfloor$ and $q_0(P_{k+1})=\lceil k/2\rceil$, by taking $L=P_{k+1}$, $t_2=0$, $r_2=\lceil k/2\rceil-2$ and $s=1$ in Lemma \ref{lem:q}. The upper bound follows from \eqref{eq:P_k} by taking $L=P_{k+1}$, $t_1=0$, $r=\lfloor k/2\rfloor-1$ and $s=k\bmod 2+1$ in Theorem \ref{thm:tP2}.
\end{proof}
\begin{urem} Note that for odd $k>4$, the upper and lower bounds we get, in Corollary \ref{cor:PktP2}, for $AR(n,P_{k+1}\cup tP_2)$ (for large enough $n$) differ only by 1.
\end{urem}

\begin{cor}\label{cor:C3tP2} For any integers $t\geq 1$ and $n>\frac{5}{2}t+6$,
$$AR(n,C_3\cup tP_2)=t\left(n-\frac{t+1}{2}\right)+2.$$
In addition, for any integers $k\geq 4$, $t\geq 0$ and $n\geq 2t+k$, 
\begin{equation*}AR(n,C_k\cup tP_2)\geq (t+\lceil k/2\rceil-2)\left(n-\frac{t+\lceil k/2\rceil-1}{2}\right)+2,\end{equation*}
and for any integers $k\geq 4$, $t\geq 0$ and $n>\frac{5}{2}t+\frac{9}{4}k-\frac{5}{4}$,
$$AR(n,C_k\cup tP_2)\leq\left(t+\frac{k}{2}+\frac{1}{k-1}-1\right)\left(n-\frac{t+\frac{k}{2}+\frac{1}{k-1}}{2}\right)+\frac{1}{2}\left(\frac{k}{2}+\frac{1}{k-1}\right)\left(\frac{k}{2}+\frac{1}{k-1}-1\right).$$
\end{cor}
\begin{proof} By \cite[Proposition 6.2]{BGR}, $AR(n,C_3\cup P_2)=n+1$ for any $n\geq 6$, and clearly $q_1(C_3)=1$ and $q_0(C_3)=2$. Taking $L=C_3$, $t_1=1$, $t_2=0$, $r=0$, $s=1$ and $n_0=6$ in Proposition \ref{prop:tP2tight} and Remark \ref{rem:gamma} we get the first claim.

For $k\geq 4$, the lower bound for $AR(n,C_k\cup tP_2)$  follows, since $q_1(C_k)=\lfloor k/2\rfloor$ and $q_0(C_k)=\lceil k/2\rceil$, by taking $L=C_k$, $t_2=0$, $r_2=\lceil k/2\rceil-2$ and $s=1$ in Lemma \ref{lem:q}. The upper bound follows by taking $L=C_k$, $t_1=0$, $r=\frac{k}{2}+\frac{1}{k-1}-1$, $s=\frac{1}{2}\left(\frac{k}{2}+\frac{1}{k-1}\right)\left(\frac{k}{2}+\frac{1}{k-1}-1\right)-1$ and $n_0=k$ in Theorem \ref{thm:tP2} and Remark \ref{rem:gamma}, since by \eqref{eq:C_k}, for any integers $n\geq k\geq 4$,  
\begin{gather*}AR(n,C_k)=\binom{k-1}{2}\left\lfloor\frac{n}{k-1}\right\rfloor+\left\lceil\frac{n}{k-1}\right\rceil+\binom{n\bmod(k-1)}{2}\leq\left(\frac{k-2}{2}+\frac{1}{k-1}\right)n=\\
=\left(\frac{k}{2}+\frac{1}{k-1}-1\right)\left(n-\frac{\frac{k}{2}+\frac{1}{k-1}}{2}\right)+\frac{1}{2}\left(\frac{k}{2}+\frac{1}{k-1}\right)\left(\frac{k}{2}+\frac{1}{k-1}-1\right).\qedhere\end{gather*}
\end{proof}

\section{The anti-Ramsey number of $L\cup tP_3$}\label{sec:tP3}

\begin{theorem}\label{thm:LtP3} Let $L$ be a graph, let $k_1\geq 0$ and $n_0\geq |V(L)|+3k_1$ be integers, and let $r$ and $s$ be real numbers. Suppose that $AR(n,L\cup k_1P_2)\leq (k_1+r)\left(n-\frac{k_1+r+1}{2}\right)+s+1$ for any integer $n\geq n_0$.

Then, there is a constant $\gamma_3$, depending only on $L$, $k_1$, $r$, $s$ and $n_0$, such that for any integers $k\geq k_1$ and $n>5k+\gamma_3$, 
\begin{equation*}\label{eq:LtP2}AR(n,L\cup kP_3)\leq (k+r)\left(n-\frac{k+r+1}{2}\right)+s+1.\end{equation*} 
\end{theorem}
\begin{proof} The proof is by induction on $k$, and follows the same path as the proof of Theorem \ref{thm:tP2}, but uses some slightly more elaborate arguments and calculations.
The base case, $k=k_1$, is obvious (provided $5k_1+\gamma_3\geq n_0-1$). 

Now let $k>k_1$, and assume that $AR(n,L\cup (k-1)P_3)\leq(k-1+r)\left(n-\frac{k-1+r+1}{2}\right)+s+1$ for any $n>5(k-1)+\gamma_3$. Let $c$ be an edge-coloring of $K^V$, where $|V|=n>5k+\gamma_3$, by at least $(k+r)\left(n-\frac{k+r+1}{2}\right)+s+1$ colours. We will find a multicoloured copy of $L\cup kP_3$ in $K^V$. Let $\ell:=|V(L)|$.
The proof is divided into two cases.
\medskip\newline
{\bf Case 1.} $d_c(v_0)\geq 3k+\ell$ for some vertex $v_0$.
\medskip

Changing the colour of every edge $e$ of $K^{V-\{v_0\}}$ for which $c(e)\in C(v_0)$ (if there are any such edges) to some common colour $c_0$, we get an edge-colouring $c^*$ of $K^{V-\{v_0\}}$ by at least 
$$(k+r)\left(n-\frac{k+r+1}{2}\right)+s+1-(n-1)=(k-1+r)\left(n-1-\frac{k-1+r+1}{2}\right)+s+1$$
colours. Since clearly $n-1>5k+\gamma_3-1>5(k-1)+\gamma_3$, we get by the induction hypothesis that $K^{V-\{v_0\}}$ contains a copy $G$ of $L\cup (k-1)P_3$ which is multicoloured with respect to $c^*$, and therefore also according to the original colouring $c$.

The vertex $v_0$ is the endpoint of at least $3k+\ell$ edges with distinct colours (with respect to $c$). The other endpoint of at most $3(k-1)+\ell$ of those edges is a vertex of $G$. Also, at most one of those edges have the same colour, according to $c$, as an edge of $G$ (since at most one of the edges of $G$ is coloured by the colour $c_0$ according to $c^*$, i.e., by a colour in $C(v_0)$ with respect to $c$). Therefore we are surely left with at least two edges, with different colours, $v_0w_1$ and $v_0w_2$, such that $w_1,w_2\notin V(G)$ and $c(v_0w_1),c(v_0w_2)\notin\{c(e)\mid e\in E(G)\}$. By adding the path $w_1v_0w_2$ to $G$ we get the desired muticoloured copy of $L\cup kP_3$.  
\medskip\newline {\bf Case 2.} $d_c(v)\leq 3k+\ell-1$ for all $v\in V$.
\medskip

By the induction hypothesis $K^V$ clearly contains a multicoloured copy, $G$, of $L\cup (k-1)P_3$. Assume, by contradiction, that $K^V$ does not contain a multicoloured copy of $L\cup kP_3$.

Form a graph $H$ on the vertex set $V$ by adding to the edges of $G$ a single edge of each colour of $c$ not represented in $G$. By our assumptions, $d_H(v)\leq 3k+\ell-1$ for any $v\in V$, and $H$ does not contain a copy of $L\cup kP_3$.   

Let $U_L$ be the vertex set of the $L$ part of $G$, $U$ the vertex set of the $(k-1)P_3$ part of $G$, and let $W:=V- (U_L\cup U)$. 
Call a vertex $u\in U$ {\bf fat} if $\lvert E_H(\{u\},W)\rvert\geq 3$, and {\bf thin} otherwise.

If $u,v\in U$ are vertices of the same path in $G$, and $u$ is fat, then $\lvert E_H(\{v\},W)\rvert\leq 1$. (Otherwise, a simple case analysis shows that we could get from $G$ a copy of $L\cup kP_3$ in $H$ by replacing the path containing $u$ and $v$ by two paths, combined of $u$, $v$, the third vertex in their path in $G$, and some neighbours in $W$ of $u$ and $v$). In particular, any path in $G$ contains at most one fat vertex.
 
Let $F\subseteq U$ be the set of fat vertices, $N\subseteq U$ the set of thin vertices that their path in $G$ contains a fat vertex, and $T\subseteq U$ the set of all other thin vertices. Notice that 
\begin{equation}\label{NFT2}|N|=2|F|,\quad |T|=3(k-1-|F|).\end{equation}

If $u, w_1, w_2\in N$ such that $w_1, w_2$ are in the same path in $G$, then $u$ is adjacent to at most one of the vertices $w_1, w_2$. (Otherwise, a simple case analysis shows that we could get from $G$ a copy of $L\cup kP_3$ in $H$ by replacing the two paths containing $u$, $w_1$ and $w_2$ by three paths, combined from the vertices of those two paths and neighbours in $W$ of their fat vertices). We therefore have,
\begin{equation}\label{ENN2}\lvert E_H(N,N)\rvert\leq \frac{1}{2}\cdot|N|\cdot\frac{|N|}{2}=|F|^2.\end{equation}
Also,
\begin{equation}\label{EWW2}\lvert E_H(W,W)\rvert\leq \frac{|W|}{2}=\frac{n-3(k-1)-\ell}{2},\end{equation} 
otherwise $E_H(W,W)$ contains at least two adjacent edges, and by adding them to $G$ we would get a copy of $L\cup kP_3$ in $H$. 
Since $\lvert E_H(\{v\},W)\rvert\leq 1$ for any $v\in N$, it follows that
\begin{equation}\label{ENW2}\lvert E_H(N,W)\rvert\leq |N|=2|F|,\end{equation}
by the definition of thin vertices,
\begin{equation}\label{ETW2}\lvert E_H(T,W)\rvert\leq 2|T|=2\cdot 3(k-1-|F|),\end{equation}
and by the assumption that $\Delta(H)\leq 3k+\ell-1$, we have
\begin{equation}\label{EFV2}\lvert E_H(F\cup U_L,V)\rvert\leq |F\cup U_L|(3k+\ell-1)-|E(L)|=(|F|+\ell)(3k+\ell-1)-|E(L)|.\end{equation} 
Combining \eqref{ENN2}, \eqref{EWW2}, \eqref{ENW2}, \eqref{ETW2} and \eqref{EFV2}, and using \eqref{NFT2} we get
\begin{align*}\lvert E_H(V,V)\rvert=&\lvert E_H(N,N)\rvert+\lvert E_H(N,T)\rvert+\lvert E_H(T,T)\rvert+\lvert E_H(N,W)\rvert+\lvert E_H(T,W)\rvert+\\
&+\lvert E_H(W,W)\rvert+\lvert E_H(F\cup U_L,V)\rvert\leq\\
\leq&\frac{1}{2}\cdot|N|\cdot\frac{|N|}{2}+|N|\cdot |T|+\binom{|T|}{2}+|N|+2|T|+\\
&+\frac{|W|}{2}+(3k+\ell-1)(|F|+\ell)-|E(L)|=\\
=&|F|^2+2|F|\cdot 3(k-1-|F|)+\binom{3(k-1-|F|)}{2}+2|F|+2\cdot 3(k-1-|F|)+\\
&+\frac{n-3(k-1)-\ell}{2}+(3k+\ell-1)(|F|+\ell)-|E(L)|=\\
=&\frac{n}{2}+\frac{1}{2}(3k+\ell-2)^2+\ell^2-|E(L)|-\frac{3}{8}-\frac{\left(|F|-(\ell-\frac{1}{2})\right)^2}{2}\leq\\
\leq&\frac{n}{2}+\frac{1}{2}(3k+\ell-2)^2+\ell^2-|E(L)|-\frac{1}{2}.\end{align*}
Therefore
$$(k+r)\left(n-\frac{k+r+1}{2}\right)+s+1\leq\frac{n}{2}+\frac{1}{2}(3k+\ell-2)^2+\ell^2-|E(L)|-\frac{1}{2}.$$
After some rearranging we get
$$(k+r-\frac{1}{2})\left(n-5k+4r-3\ell+3\right)\leq \frac{1}{2}(\ell-3r-\frac{1}{2})^2+\ell^2-|E(L)|-s-\frac{9}{8},$$
yielding a contradiction for $n>5k+\gamma_3$, if we take $\gamma_3$ such that
\begin{equation*}(k_1+1+r-\frac{1}{2})\left(\gamma_3+4r-3\ell+3\right)\geq\max\left\{0\,,\,\frac{1}{2}(\ell-3r-\frac{1}{2})^2+\ell^2-|E(L)|-s-\frac{9}{8}\right\}.\qedhere\end{equation*}
\end{proof}
\begin{rem}\label{rem:gamma3} As the proof above shows, $\gamma_3$ may be taken to be 
\begin{equation*}\max\left\{n_0-1-5k_1\,,\,\left\lfloor3\ell-4r-3+\frac{\frac{1}{2}(\ell-3r-\frac{1}{2})^2+\ell^2-|E(L)|-s-\frac{9}{8}}{k_1+r+\frac{1}{2}}\right\rfloor\right\}\end{equation*}
if $\frac{1}{2}(\ell-3r-\frac{1}{2})^2+\ell^2-|E(L)|-s-\frac{9}{8}\geq 0$, and $\max\left\{n_0-1-5k_1\,,\,\left\lceil 3\ell-4r-4\right\rceil\right\}$ otherwise.
\end{rem}
\medskip
We now show some consequences of Theorem \ref{thm:LtP3}.

\begin{cor}\label{cor:tP3}  
For any integers $k\geq 2$ and $n>5k+1$,
\begin{equation}\label{eq:kP3}AR(n,kP_3)=(k-1)\left(n-\frac{k}{2}\right)+2.\end{equation}
\end{cor}
\begin{proof}
The lower bound follows, since clearly $q_1(kP_3)=k$, by taking $G=kP_3$, $r_1=k-1$ and $s=1$ in Lemma \ref{lem:q}. 
Since $AR(n,2P_3)=n+1$ for any $n\geq 7$ by \cite[Proposition 6.6]{BGR}, the upper bound follows by taking $L$ to be the empty graph, $k_1=2$, $r=-1$, $s=1$ and $n_0=7$ in Theorem \ref{thm:LtP3} and Remark \ref{rem:gamma3}.
\end{proof}

\begin{urems}
\begin{itemize}
\item
Using only the trivial observation that $AR(n,P_3)=2$ for any $n\geq 3$, instead of Proposition 6.5 in \cite{BGR}, we can get, by taking $L$ to be the empty graph, $k_1=1$, $r=-1$, $s=1$ and $n_0=3$ in Theorem \ref{thm:LtP3} and Remark \ref{rem:gamma3}, that \eqref{eq:kP3} holds for any integers $k\geq 1$ and $n>5k+3$.
\item
Note that \eqref{eq:kP3} does not hold for $3k\leq n<5k-3$. Indeed, colouring all the edges between $3k-2$ chosen vertices of $K_n$ by distinct colours, and all other edges of $K_n$ by one additional colour, we get that for any positive integers $n\geq 3k$,
\begin{equation*}\label{eq:tP_3small_n} AR(n,kP_3)\geq\frac{(3k-2)(3k-3)}{2}+2=(k-1)\left(\frac{9}{2}k-3\right)+2,\end{equation*}
which is larger than $(k-1)\left(n-\frac{k}{2}\right)+2$ if $n<5k-3$.
We suspect that for any integers $k\geq 1$, $n\geq 3k$,
\begin{equation*}\label{eq:tP3:lowerbound}AR(n,kP_3)=(k-1)\max\left\{n-\frac{k}{2}\,,\,\frac{9}{2}k-3\right\}+2=\begin{cases}(k-1)\left(n-\frac{k}{2}\right)+2&n\geq 5k-3\\(k-1)\left(\frac{9}{2}k-3\right)+2& n<5k-3~.\end{cases}\end{equation*}
\end{itemize}
\end{urems}

\begin{cor}\label{cor:PtkP3} 
For any integer $t\geq 3$ there is a constant $\gamma_3(P_{t+1})$ such that for any integers $k\geq 0$ and $n>5k+\gamma_3(P_{t+1})$,
$$AR(n,P_{t+1}\cup kP_3)=(k+\lfloor t/2\rfloor-1)\left(n-\frac{k+\lfloor t/2\rfloor}{2}\right)+2+t\bmod 2.$$
\end{cor}
\begin{proof}
The lower bound follows, since $q_1(P_{t+1}\cup kP_3)=\lceil t/2\rceil+k$ and $q_2(P_{t+1}\cup kP_3)=\lceil t/2\rceil+k-1$, by taking $G=P_{t+1}\cup kP_3$, $r_1=\lfloor t/2\rfloor+k-1$ and $s=1+t\bmod 2$ in Lemma \ref{lem:q}.
The upper bound follows from \eqref{eq:P_k} by taking $L=P_{t+1}$, $k_1=0$, $r=\lfloor t/2\rfloor-1$ and $s=1+t\bmod 2$ in Theorem \ref{thm:LtP3}.
\end{proof}

\begin{cor}\label{cor:P2kP3} 
For any integers $k\geq 1$ and $n>5k+27$,
$$AR(n,P_2\cup kP_3) =(k-1)\left(n-\frac{k}{2}\right)+3.$$
\end{cor}
\begin{proof}
The lower bound follows, since $q_2(P_2\cup kP_3)=k$, by taking $G=P_2\cup kP_3$, $r_1=k-1$ and $s=2$ in Lemma \ref{lem:q}. The upper bound follows by taking $L=P_2$, $k_1=1$, $r=-1$, $s=2$ and $n_0=5$ in Theorem \ref{thm:LtP3} and Remark \ref{rem:gamma3},
since $AR(n,P_2\cup P_3)=3$ for any $n\geq 5$ (\cite[Proposition 3.3]{BGR}).
\end{proof}

\begin{cor}\label{cor:tP2kP3}
For any integers $t\geq 2$, $k\geq 2$ and $n>\min\{5k+\frac{13}{2}t+8\,,\,\frac{5}{2}t+\frac{19}{2}k+7\}$,
\begin{equation}\label{kP2tP3}AR(n,tP_2\cup kP_3) =(k+t-2)\left(n-\frac{k+t-1}{2}\right)+2.\end{equation}
\end{cor}
\begin{proof}
Since $q_1(P_2\cup kP_3)=k$ and $q_0(P_2\cup kP_3)=k+1$, we get, by taking $L=kP_3$, $t_2=1$, $r_2=k-2$ and $s=1$ in Lemma \ref{lem:q}, that the lower bound in \eqref{kP2tP3} holds for any $t\geq 1$, $k\geq 1$ and $n\geq 2t+3k$. 

Since $AR(n,P_3\cup tP_2) =(t-1)\left(n-\frac{t}{2}\right)+2$ for any $t\geq 2$ and $n>\frac{5}{2}t+12$, by Corollary \ref{cor:PktP2}, we get, by taking $L=tP_2$, $k_1=1$, $r=t-2$, $s=1$ and $n_0=\lfloor\frac{5}{2}t+13\rfloor$ in Theorem \ref{thm:LtP3} and Remark \ref{rem:gamma3}, that the upper bound in \eqref{kP2tP3} holds for any $t\geq 2$, $k\geq 1$ and $n>5k+\frac{13}{2}t+8$.

In particular, $AR(n,2P_2\cup kP_3) \leq k\left(n-\frac{k+1}{2}\right)+2$ for any $k\geq 1$ and $n>5k+21$, so by taking $L=kP_3$, $t_1=2$, $r=k-2$, $s=1$ and $n_0=5k+22$ in Theorem \ref{thm:tP2} and Remark \ref{rem:gamma}, we get that the upper bound in \eqref{kP2tP3} holds for any $k\geq 1$, $t\geq 2$ and $n>\frac{5}{2}t+\max\{5k+16\,,\,\frac{19}{2}k+3\}$ (hence the upper bound in \eqref{kP2tP3} holds for any $t\geq 2$, $k\geq 2$ and $n>\frac{5}{2}t+\frac{19}{2}k+7$).
\end{proof}


\begin{thebibliography}{99}

\bibitem{A}
N. Alon, On a conjecture of Erd\H os, Simonovits, and S\'os concerning anti-Ramsey theorems, J. Graph Theory {\bf 7} (1983), no.~1, 91--94. 

\bibitem{BGR} A. Bialostocki, S. Gilboa \ and\ Y. Roditty, Anti-Ramsey numbers of small graphs, submitted

\bibitem{ESS}
P. Erd\H os, M. Simonovits\ and\ V. T. S\'os, Anti-Ramsey theorems, in {\it Infinite and finite sets (Colloq., Keszthely, 1973; dedicated to P. Erd\H os on his 60th birthday), Vol. II}, 633--643. Colloq. Math. Soc. J\'anos Bolyai, 10, North-Holland, Amsterdam.

\bibitem{FKSS}
S. Fujita, A. Kaneko, I. Schiermeyer\ and\ K. Suzuki, A rainbow $k$-matching in the complete graph with $r$ colors, Electron. J. Combin. {\bf 16} (2009), no.~1, Research Paper 51, 13 pp.

\bibitem{HY}
R. Haas\ and\ M. Young, The anti-Ramsey number of perfect matching, Discrete Math. {\bf 312} (2012), no.~5, 933--937.

\bibitem{MbNl}
J. J. Montellano-Ballesteros\ and\ V. Neumann-Lara, An anti-Ramsey theorem on cycles, Graphs Combin. {\bf 21} (2005), no.~3, 343--354. 

\bibitem{S}
I. Schiermeyer, Rainbow numbers for matchings and complete graphs, Discrete Math. {\bf 286} (2004), no.~1-2, 157--162. 

\bibitem{SS}
M. Simonovits\ and\ V. T. S\'os, On restricted colorings of $K\sb n$, Combinatorica {\bf 4} (1984), no.~1, 101--110.

\end{thebibliography}
\end{document}